\documentclass[reqno]{amsart}
\usepackage[sc]{mathpazo}

\usepackage{amsmath}
\usepackage{amsthm}
\usepackage{amsfonts}
\usepackage{graphicx,psfrag,epsfig}
\usepackage{color}
\usepackage{mathrsfs}
\usepackage{pdfsync}
\usepackage{cite}

\newtheorem{thm}{Theorem}[section]
\newtheorem{lem}[thm]{Lemma}
\newtheorem{cor}[thm]{Corollary}
\newtheorem{prop}[thm]{Proposition}

\newtheorem{rem}[thm]{Remark}

\DeclareMathAlphabet{\mathpzc}{OT1}{pzc}{m}{it}

\numberwithin{equation}{section}

\newcommand{\bqn}{\begin{equation}}
\newcommand{\eqn}{\end{equation}}
\newcommand{\bqnn}{\begin{equation*}}
\newcommand{\eqnn}{\end{equation*}}

\newcommand{\R}{\mathbb{R}}

\newcommand{\C}{\mathbb{C}}

\newcommand{\B}{D}
\newcommand{\ml}{\mathcal{L}}

\newcommand{\ve}{\varepsilon}

\newcommand{\rd}{\mathrm{d}}

\newcommand{\dhr}{\mathrel{\lhook\joinrel\relbar\kern-.8ex\joinrel\lhook\joinrel\rightarrow}} 

\title[A free boundary problem for MEMS with nonlinear bending]
{A free boundary problem modeling electrostatic MEMS:\\ II. nonlinear bending effects}

\author{Philippe Lauren\c{c}ot}
\address{Institut de Math\'ematiques de Toulouse, CNRS UMR~5219, Universit\'e de Toulouse \\ F--31062 Toulouse Cedex 9, France}
\email{laurenco@math.univ-toulouse.fr}

\author{Christoph Walker}
\address{Leibniz Universit\"at Hannover\\ Institut f\" ur Angewandte Mathematik \\ Welfengarten 1 \\ D--30167 Hannover\\ Germany}
\email{walker@ifam.uni-hannover.de}

\begin{document}

\date{\today}

\begin{abstract}
Well-posedness of a free boundary problem for electrostatic microelectromechanical systems (MEMS) is investigated when nonlinear bending effects are taken into account. The model describes the evolution of the deflection of an electrically conductive elastic membrane suspended above a fixed ground plate together with the electrostatic potential in the free domain between the membrane and the fixed ground plate. The electrostatic potential is harmonic in that domain and its values are held fixed along the membrane and the ground plate. The equation for the membrane deflection is a parabolic quasilinear fourth-order equation, which is coupled to the gradient trace of the electrostatic potential on the membrane.
\end{abstract}

\keywords{MEMS, free boundary problem, curvature, well-posedness, bending}
\subjclass[2010]{35R35, 47D06, 35M33, 35Q74, 35K93, 74M05}

\maketitle

\section{Introduction}\label{s1}

Microelectromechanical systems (MEMS) are miniaturized structures that combine logic elements and micromechanical components, often acting as sensors or actuators. These tiny devices enable a myriad of applications and are ubiquitous in a vast range of nowadays electronics like optical switches, micropumps, micromirrors, displays, or audio components. Idealized modern MEMS often consist of two components: a rigid ground plate and an electrically conductive thin elastic membrane that is held fixed along its boundary above the rigid plate. The design of such devices is based on the interaction between electrostatic and elastic forces. Indeed, applying a voltage difference between the two components generates a Coulomb force which induces displacements of the membrane and thus transforms electrostatic energy into mechanical energy. There is, however, an upper limit for the applied voltage beyond which the electrostatic force cannot be balanced by the elastic response of the membrane and the membrane then touches down on the rigid plate. This phenomenon is usually referred to as \emph{``pull-in''} instability or \emph{touchdown}. Estimating this pull-in instability threshold is an important issue in applications as it determines the optimal operating conditions of the MEMS device.

\medskip

The mathematical description of idealized MEMS devices involves the deflection of the deformable membrane above the ground plate and the electrostatic potential in between. From the energy balance for the membrane, one may derive the equation governing its dynamics, which involves the gradient of the potential on the membrane. As the potential is harmonic in the region between the ground plate and the membrane with given values on these two components, one is thus naturally led to a free boundary problem, see, for example, \cite{P02,BP03,PT01} and the references therein. However, most mathematical analysis so far has  been dedicated to simplified variants thereof for which we refer to the next section.
 
In this paper we investigate the free boundary problem and specifically take into account bending effects which result in a fourth-order equation for the membrane deflection. Moreover, different than in most research hitherto, which was restricted to small deformations from the outset, we shall not neglect curvature effects and hence obtain a quasilinear fourth-order equation for the membrane deflection. As pointed out in \cite{BrubakerPelesko_EJAM}, retaining gradient terms might affect the value of the pull-in voltage and is thus important in applications. We will be more precise in the following section, where we derive the model. Our main results regarding the local and global well-posedness in dependence on the applied voltage difference $\lambda$ are stated in Section~\ref{s3}, where we also present the existence of steady-state solutions for small voltage values. The corresponding proofs are then contained in the subsequent sections.

\section{Derivation of the Model}\label{s2}

We begin this section with a review of the free boundary model for electrostatic MEMS, when bending effects are taken into account. We basically follow the derivation performed in \cite[Chapter~7]{BP03} but without the \textit{a priori} assumption of small deformations. 

\subsection{The Model}

\begin{figure}
\centering\includegraphics[scale=.7]{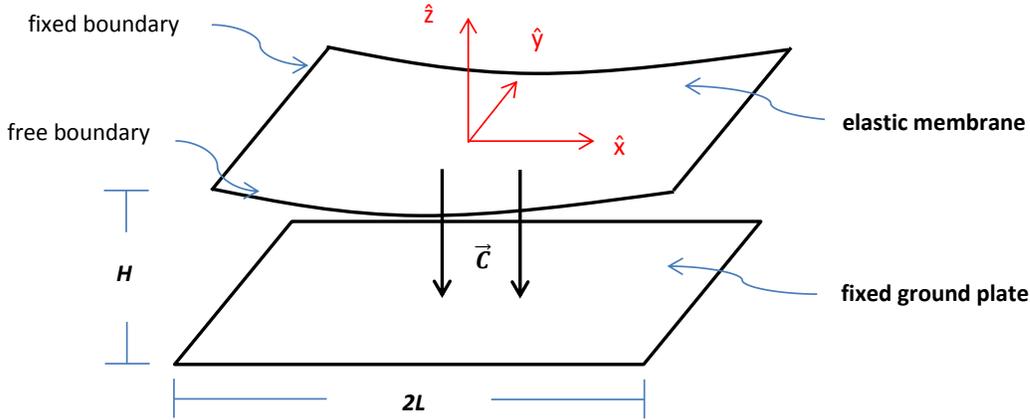}
\caption{\small Idealized electrostatic MEMS device.}\label{MEMS1b}
\end{figure}

We consider a rectangular thin elastic membrane that is coated with a thin dielectric film and suspended above a rigid plate. The $(\hat{x},\hat{y},\hat{z})$-coordinate system is chosen such that the ground plate of dimension $[-L,L]\times [-l,l]$ in the $(\hat{x},\hat{y})$-direction is located at $\hat{z}=-H$, while the undeflected membrane with the same dimension $[-L,L]\times [-l,l]$ in the $(\hat{x},\hat{y})$-direction is located at $\hat{z}=0$. The membrane is held fixed along the edges in the $\hat{y}$-direction while the edges in the $\hat{x}$-direction are free. The situation is illustrated in Figure~\ref{MEMS1b}. Assuming homogeneity in the $\hat{y}$-direction, the membrane may thus be considered as an elastic strip and the $\hat{y}$-direction is omitted in the sequel. Holding the strip at potential $V$ while the rigid plate is grounded induces a Coulomb force across the device which causes a mechanical deflection of the strip. We let $\hat{u}=\hat{u}(\hat{t},\hat{x}) >-H$ denote the deflection of the strip at the point $\hat{x}\in (-L,L) $ and time $\hat{t}$, and we let $\hat{\psi}=\hat{\psi}(\hat{t},\hat{x} ,\hat{z})$ denote the electrostatic potential at the point $(\hat{x},\hat{z}) $ and time $\hat{t}$. We suppress time $\hat{t}$ for the moment.

\medskip

The electrostatic potential $\hat{\psi}$ is harmonic in the region
$$
\hat{\Omega}(\hat{u}):=\left\{(\hat{x},\hat{z})\,;\, -L<\hat{x}<L\,,\, -H<\hat{z}<\hat{u}(\hat{x})\right\}
$$
between the ground plate and the strip, that is,
\begin{equation}\label{psihat}
\Delta\hat{\psi}=0 \quad\text{in}\quad \hat{\Omega}(\hat{u})\ ,
\end{equation}
and is subject to the boundary conditions
\begin{equation}\label{psihatbc}
\hat{\psi}(\hat{x},-H)=0\ ,\quad \hat{\psi}(\hat{x},\hat{u}(\hat{x}))= V\ ,
 \qquad \hat{x}\in(-L,L)\ .
\end{equation}
The {\it electrostatic energy} $\hat{E}_e$ in dependence of the deflection $\hat{u}$  is given by
\begin{equation*} 
\hat{E}_e(\hat{u})=-\frac{\epsilon_0}{2}\int_{-L}^L\int_{-H}^{\hat{u}(\hat{x})}\vert\nabla\hat{\psi}(\hat{x},\hat{z})\vert^2\,\,\rd\hat{z}\,\rd \hat{x}\,\ ,
\end{equation*} 
with $\epsilon_0$ being the permittivity of free space.  The \textit{surface energy} $\hat{E}_s$ in dependence of the deflection $\hat{u}$ is proportional to the tension $T$ and to the change of arc length of the strip, i.e.
\begin{equation*} 
\hat{E}_s(\hat{u})=T\int_{-L}^L  \left(\sqrt{1+(\partial_{\hat{x}}\hat{u}(\hat{x}))^2}-1\right)\, \rd \hat{x}\ .
\end{equation*}
Letting $Y$ denote Young's modulus and $I$ the momentum of inertia, the \textit{bending energy}~$\hat{E}_b$ is
\begin{equation*} 
\hat{E}_b(\hat{u})=\frac{YI}{2}\int_{-L}^L  \left\vert\partial_{\hat{x}}\left(\frac{\partial_{\hat{x}}\hat{u}(\hat{x})}{\sqrt{1+(\partial_{\hat{x}}\hat{u}(\hat{x}))^2}}\right)\right\vert^2\, \sqrt{1+(\partial_{\hat{x}}\hat{u}(\hat{x}))^2}\,\rd \hat{x}\ ,
\end{equation*}
where
$$
\partial_{\hat{x}}\left(\frac{\partial_{\hat{x}}\hat{u}}{\sqrt{1+(\partial_{\hat{x}}\hat{u})^2}}\right)=\frac{\partial_{\hat{x}}^2\hat{u}}{\left(1+(\partial_{\hat{x}}\hat{u})^2\right)^{3/2}}
$$
is the curvature of the graph of $\hat{u}$ and $\rd s=\sqrt{1+(\partial_{\hat{x}}\hat{u}(\hat{x}))^2}\rd \hat{x}$ its arc length element. The \textit{total energy} of the system is then the sum $\hat{E}(\hat{u}) := \hat{E}_e(\hat{u}) + \hat{E}_s(\hat{u}) + \hat{E}_b(\hat{u})$. Finally, the strip is clamped at its ends, so that $\hat{u}(\pm L) = \partial_{\hat{x}} \hat{u}(\pm L) = 0$.

\medskip

We next introduce the dimensionless variables
$$
x=\frac{\hat{x}}{L}\ ,\quad z=\frac{\hat{z}}{H}\ ,\quad u=\frac{\hat{u}}{H}\ ,\quad \psi = \frac{\hat{\psi}}{V}\ ,
$$
and denote the aspect ratio of the device by $\ve:=H/L$. In these dimensionless variables, we may then write the total energy $E$ in dependence of the deflection $u$ in the form
\begin{equation}\label{E}
\begin{split}
E(u) = &\frac{YI\ve^2}{2L}\int_{-1}^1  \left\vert\partial_{x}\left(\frac{\partial_{x}u(x)}{\sqrt{1+\ve^2(\partial_{x}u(x))^2}}\right) \right\vert^2\, \sqrt{1+\ve^2(\partial_{x}u(x))^2}\,\rd x \\
& + TL\int_{-1}^1  \left(\sqrt{1+\ve^2(\partial_{x}u(x))^2}-1\right)\, \rd x \\
& - \frac{\epsilon_0 V^2}{2\ve}\int_{\Omega(u)}\left(\ve^2\vert\partial_{x}\psi(x,z)\vert^2+\vert\partial_z\psi(x,z)\vert^2\right)\,\rd (x,z)\ ,
\end{split}
\end{equation} 
with
$$
\Omega(u)=\left\{(x,z)\,;\, -1<x<1\,,\, -1<z<u(x)\right\}\ .
$$
The equilibrium configurations of the device are the critical points of the total energy $E$ and are given by the solutions to the corresponding Euler-Lagrange equation reading
\begin{equation}\label{EL}
\begin{split}
0=\ & -\ve^2\,\beta\, \partial_{x}^2\left(\frac{\partial_{x}^2u}{(1+\ve^2(\partial_{x}u)^2)^{5/2}}\right)- \frac{5}{2}\,\ve^4\,\beta\, \partial_{x} \left( \frac{\partial_xu(\partial_{x}^2u)^2}{(1+\ve^2(\partial_{x}u)^2)^{7/2}} \right)\\
&+\ve^2\,\tau\, \partial_{x}\left(\frac{\partial_{x}u}{(1+\ve^2(\partial_{x}u)^2)^{1/2}}\right)
-\ve^2\,\lambda \left(\ve^2\vert\partial_{x}\psi(x,u(x))\vert^2+\vert\partial_z\psi(x,u(x))\vert^2\right)
\end{split}
\end{equation}
for $x\in (-1,1)$, where we have set
$$
\beta:= \frac{YI}{L}\ ,\qquad
\tau:=TL\ ,\qquad
\lambda=\lambda(\ve):=\frac{\epsilon_0V^2}{2\ve^3}\ .
$$
While the derivations of the first three terms in \eqref{EL} from the bending and stretching energies in \eqref{E} follow by classical arguments, the derivation of the last term in \eqref{EL} from the electrostatic energy in \eqref{E} is more involved and relies on shape optimization techniques, see \cite[Section~5.3]{HP05} for instance (recall that $\psi$ depends non-locally on $u$ according to \eqref{psihat}).

\medskip

For the dynamics of the membrane deflection $u=u(\hat{t},x)$, it follows from Newton's second law that the inertia $\rho h_0 \partial_{\hat{t}}^2 u$ (with $\rho$ and $h_0$ denoting, respectively, the membrane's mass density per unit volume and thickness) balances the elastic and electrostatic forces, given by the right hand side of equation~\eqref{EL}, and we further account for a damping force $a\partial_{\hat{t}} u$ which is linearly proportional to the velocity. Scaling time based on the strength of damping according to $t=\hat{t}\ve^2/a$, setting $\gamma^2:=\rho h_0 \ve^2/a^2$, and introducing the quasilinear fourth-order operator 
\begin{equation}\label{AA}
\begin{split}
\mathcal{K}(u):=\beta\, & \partial_{x}^2\left(\frac{\partial_{x}^2u}{(1+\ve^2(\partial_{x}u)^2)^{5/2}}\right) +  \frac{5}{2}\,\ve^2\,\beta\, \partial_{x}\left(\frac{\partial_xu(\partial_{x}^2u)^2}{(1+\ve^2(\partial_{x}u)^2)^{7/2}}\right)\\
&
-\tau\, \partial_{x}\left(\frac{\partial_{x}u}{(1+\ve^2(\partial_{x}u)^2)^{1/2}}\right)
\end{split}
\end{equation}
allow us to write the damping dominated evolution for the strip deflection in the form
\begin{equation}\label{u}
\begin{split}
\gamma^2 \partial_t^2 u + \partial_t u+ \mathcal{K}(u)= - \lambda \left( \ve^2 \vert \partial_{x} \psi(t,x,u(t,x)) \vert^2 + \vert\partial_z\psi(t,x,u(t,x)) \vert^2\right)
\end{split}
\end{equation}
for $t>0$ and $x\in (-1,1)$, subject to the clamped boundary conditions
\begin{equation}\label{bcu}
u(t,\pm 1)= \partial_x u(t,\pm 1)=0\ ,\quad t>0\ ,
\end{equation}
and the initial condition
\begin{equation}\label{ic}
u(0,x)=u^0(x)\ ,\quad x\in (-1,1)\ .
\end{equation}
Equations \eqref{psihat}-\eqref{psihatbc} in dimensionless variables read
\begin{equation}\label{psi}
\ve^2 \,\partial_x^2\psi+\partial_z^2\psi=0 \ ,\quad (x,z)\in \Omega(u(t))\ ,\quad t>0\ ,
\end{equation}
subject to the boundary conditions (linearly extended on the lateral boundaries)
\begin{equation}\label{psibc}
\psi(t,x,z)=\frac{1+z}{1+u(t,x)}\ ,\quad (x,z)\in\partial\Omega(u(t))\ ,\quad t>0\ .
\end{equation}
The situation for \eqref{psi}-\eqref{psibc} is depicted in Figure~\ref{fig1}. Let us emphasize here that the above model is only meaningful as long as the strip does not touch down on the ground plate, that is, the deflection $u$ satisfies $u>-1$. This fact not only shows up in the definition of $\Omega(u)$ which becomes disconnected if $u$ reaches the value $-1$ at some point, but also in the right hand side of \eqref{u} which becomes singular at such points since $\psi=1$ along $z=u(x)$ with $\psi=0$ at $z=-1$. This singularity is somehow tuned by the parameter $\lambda$, which is proportional to the square of the applied voltage difference and actually governs the global well-posedness and existence of steady-state solutions for Equations \eqref{u}-\eqref{psibc}. More precisely, it is supposed that, above a certain critical threshold of $\lambda$, solutions to \eqref{u}-\eqref{psibc} cease to exist globally in time and that there are no longer steady-state solutions. 

Before stating our results on the well-posedness of Equations \eqref{u}-\eqref{psibc}, we first consider two simplified versions thereof for which some of the just mentioned physically plausible features are known to hold.

\begin{figure}
\centering\includegraphics[width=10cm]{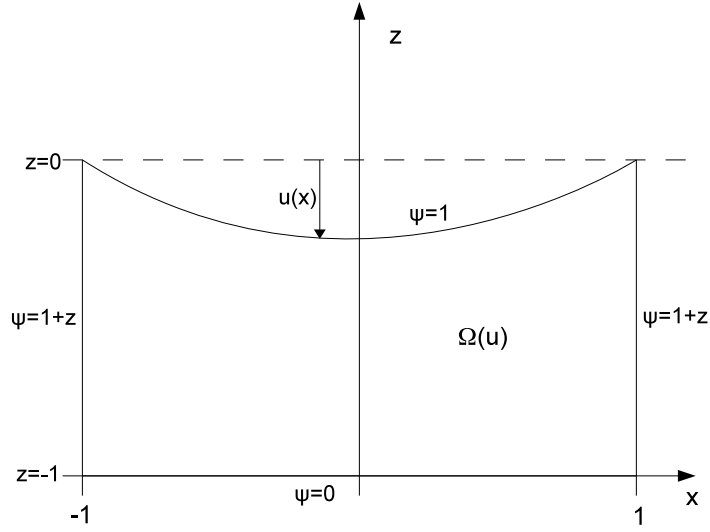}
\caption{\small Idealized electrostatic MEMS device in one dimension.}\label{fig1}
\end{figure}

\subsection{Small Deformation Model} 

Models for MEMS taking into account bending and thus including fourth-order operators in space, have mainly been investigated  with linearized curvature terms, which correspond to the \textit{a priori} assumption of small deformations. In this case, stretching and bending energies are replaced with
\begin{equation*} 
{\hat E}_s(\hat{u}) = \frac{T}{2} \int_{-L}^L  \left \vert \partial_{\hat{x}} \hat{u}(\hat{x}) \right \vert^2\, \rd \hat{x}\ ,\qquad {\hat E}_b(\hat{u}) = \frac{YI}{2} \int_{-L}^L  \left \vert \partial_{\hat{x}}^2 \hat{u}(\hat{x}) \right \vert^2\,\rd \hat{x}\ ,
\end{equation*}
and the resulting dimensionless evolution problem~\eqref{u} reduces to a nonlocal semilinear equation
\begin{equation}\label{ulinear}
\begin{split}
\gamma^2 \partial_t^2 u + \partial_t u+ \beta\, \partial_{x}^4 u -\tau\, \partial_x^2u = - \lambda \left( \ve^2 \vert \partial_{x} \psi(t,x,u(t,x)) \vert^2 + \vert \partial_z \psi(t,x,u(t,x)) \vert^2 \right)
\end{split}
\end{equation}
for $t>0$ and $x\in (-1,1)$. In \cite{LW_AnalPDE}, a rather comprehensive investigation of existence and nonexistence issues for this small deformation model may be found. We also refer to \cite{ELW1,Ci07,LaurencotWalker_ARMA} for similar results for the second-order case without bending and inertia, i.e. when $\beta=\gamma=0$.

\subsection{Small Aspect Ratio Model} 

A common assumption made in the mathematical analysis hitherto is a vanishing aspect ratio $\varepsilon=H/L$ that reduces the free boundary problem to a single equation with a right hand side involving a singularity when the strip touches down on the ground plate. More precisely, setting formally $\varepsilon=0$ allows one to solve \eqref{psi}-\eqref{psibc} explicitly for the potential $\psi=\psi_0$, that is,
\begin{equation}\label{psi0}
\psi_0(t,x,z)=\frac{1+z}{1+u_0(t,x)}\ ,\quad (t,x,z)\in [0,\infty)\times (-1,1)\times(-1,0)\ ,
\end{equation}
where the displacement $u=u_0$ now satisfies 
the so-called small aspect ratio model 
\begin{equation}\label{z0}
\begin{array}{rlll}
\gamma^2 \partial_t^2 u_0 + \partial_t u_0 +\beta\partial_x^4 u_0- \tau\partial_x^2 u_0 &\!\!\!=\!\!\!& - \displaystyle{\frac{\lambda}{(1+u_0)^2}}\,,\quad & x\in (-1,1)\,, \quad t>0\,, \\
u_0(t,\pm1) &\!\!\!=\!\!\!& 0\,, &   t>0\,, \\
u_0(0,x)&\!\!\!=\!\!\!& u^0(x)\,, & x\in (-1,1)\, .
\end{array}
\end{equation}
We shall also point out that, besides clamped boundary conditions, other boundary conditions for $u$ have been considered  in the linear case, e.g. pinned boundary conditions
$$
u(t,\pm 1)= \partial_x^2 u(t,\pm 1)=0\ ,\quad t>0\ .
$$
For details and the present state of the art on this small gap model we refer to \cite{EspositoGhoussoubGuo, Gu10, GW09, KLNT11, LinYang, LW_AHP, LL12} and the references therein.

\begin{rem}\label{rehp}
The above mentioned models \eqref{u}-\eqref{psibc}, \eqref{bcu}-\eqref{ulinear}, and \eqref{z0} are quasilinear or semilinear hyperbolic equations with a singular right hand side. When viscous or damping forces dominate over inertial forces, a commonly made assumption is to neglect second-order time derivatives and set $\gamma=0$. Then \eqref{u}-\eqref{psibc}, \eqref{bcu}-\eqref{ulinear}, and \eqref{z0} become parabolic equations and most of the mathematical analysis performed so far is devoted to this particular case. However, the case $\gamma>0$ has been studied in a few papers, see \cite{CFT,Gu10,KLNT11,LW_AHP}.
\end{rem}

\section{Main Results}\label{s3}

The main difficulties in studying Equations \eqref{u}-\eqref{psibc} lie in the non-local and singular dependence of the electrostatic potential $\psi$ on the membrane deflection $u$ together with the quasilinear structure of the operator $\mathcal{K}$ for \eqref{AA}. In combination with the hyperbolic term $\gamma^2\partial_t^u$, the well-posedness seems far from being obvious.
We thus assume from now on that damping forces are much stronger than inertial forces and neglect the latter by setting $\gamma=0$. To shorten notation, we let $I:=(-1,1)$.

\begin{thm}[{\bf Local and Global Well-Posedness}]\label{Aquasi}
Let $\gamma=0$. Consider an initial condition $u^0\in  H^{4}(I)$ satisfying the boundary conditions $u^0(\pm 1)=\partial_x u^0(\pm 1)=0$ such that $u^0(x)>-1$ for $x\in I$. Then, the following are true:

\begin{itemize}

\item[(i)] For each voltage value $\lambda>0$, there is a unique solution $(u,\psi)$ to \eqref{u}-\eqref{psibc} on the maximal interval of existence $[0,T_m)$ in the sense that
$$
u\in C^1\big([0,T_m),L_2(I)\big)\cap  C\big([0,T_m), H^{4}(I)\big)
$$
satisfies \eqref{u}-\eqref{ic} together with
$$
u(t,x)>-1\ ,\quad (t,x)\in [0,T_m)\times I\ , 
$$ 
and $\psi(t)\in H^2\big(\Omega(u(t))\big)$  solves \eqref{psi}-\eqref{psibc} in $\Omega(u(t))$ for each $t\in [0,T_m)$. 

\item[(ii)] If, for each $T>0$, there is $\kappa(T)\in (0,1)$ such that $$\|u(t)\|_{H^{4}(I)}\le \kappa(T)^{-1}\ ,\qquad u(t)\ge -1+\kappa(T)\ \text{ in } I$$ for $t\in [0,T_m)\cap [0,T]$, then the solution exists globally in time, that is, $T_m=\infty$.

\item[(iii)]  Given $\kappa\in (0,1)$, there are $\lambda_*(\kappa)>0$ and $M(\kappa)>0$ such that the solution exists globally in time provided that $\lambda\in (0,\lambda_*(\kappa))$ and $\|u^0\|_{H^{4}(I)}\le M(\kappa)$. Moreover, in this case  $u\in L_\infty(0,\infty;H^{4}(I))$ with $u(t)\ge -1+\kappa$ in $I$ for $t\ge 0$. 
\end{itemize}
\end{thm}

Theorem~\ref{Aquasi} is a somewhat paraphrased version of our actual results. We refer to the next section, in particular, to Proposition~\ref{Alocal} and Corollary~\ref{Aglobalc}, for more precise statements under weaker assumptions. Note that part~(iii) of this theorem provides global solutions for small~$\lambda$, which do not touch down on the ground plate, not even in infinite time.

\begin{rem} \label{R1a} 
In contrast to the semilinear case \eqref{bcu}-\eqref{ulinear} treated in \cite{LW_AnalPDE}, where only linear bending effects were taken into account so that the operator $\mathcal{K}$ defined in \eqref{AA} reduces to a linear operator, the global existence issue in Theorem~\ref{Aquasi} requires an additional smallness condition on $u^0$ in the $H^{4}(I)$-norm. This seems to be natural since in the quasilinear model \eqref{u}-\eqref{psibc}, the deflection may cease to be a graph for large times in which case the model itself is no longer valid. Therefore, the occurrence of a finite time singularity in the quasilinear model could correspond to a touchdown of the strip on the ground plate or a blowup of some norm of $u$. This is not the case in the semilinear model, where a finite time singularity always corresponds to a touchdown of the strip on the ground plate as shown in \cite{LW_AnalPDE}.
\end{rem}

Remarkably, when inertia and bending is neglected (i.e. $\beta=\gamma=0$) and $\mathcal{K}$ thus reduces to a (quasilinear) second-order operator, solutions to \eqref{u}-\eqref{psibc} cease to exist globally in time for large values of~$\lambda$, see \cite{ELW1,ELW3}. 

\medskip

Next, we consider steady-state solutions:

\begin{thm}[{\bf Steady-State Solutions}]\label{TStable3}
There is $\lambda_s=\lambda_s(\ve)>0$ such that for each $\lambda\in (0,\lambda_s)$ there exists a locally asymptotically stable steady state $(U_\lambda,\Psi_\lambda)$ to \eqref{u}-\eqref{psibc} with $U_\lambda\in H^4(I)$ and  $\Psi_\lambda\in H^2(\Omega(U_\lambda))$ satisfying $-1<U_\lambda< 0$ in $I$.
\end{thm}

Theorem~\ref{TStable3} is stated more precisely and in more detail in Proposition~\ref{TStable} below. If only small deformations are taken into account or if bending is neglected, then there are no steady states for large values of $\lambda$, see \cite{ELW2,LW_AnalPDE}.

\section{Well-Posedness}\label{s4} 

We aim at formulating \eqref{u}-\eqref{psibc} with $\gamma=0$ as a quasilinear Cauchy problem  only involving the deflection $u$ in an appropriate functional setting. For that purpose, let the (subspaces of) Bessel potential spaces $H_\B^{4\theta}(I)$ including clamped boundary conditions, if meaningful, be defined by
$$
H_\B^{4\theta}(I):=\left\{\begin{array}{lll}
& \{v\in H^{4\theta}(I)\,;\, v(\pm 1)=\partial_x v(\pm 1)=0\}\ , & 4\theta>\dfrac{3}{2}\ ,\\
& & \\
& \{v\in H^{4\theta}(I)\,;\, v(\pm 1)=0\}\ , & \dfrac{1}{2}<4\theta<\dfrac{3}{2}\ ,\\
& & \\
&   H^{4\theta}(I)\ , & 4\theta<\dfrac{1}{2}\ .
\end{array}
\right.
$$
Note that (for example, see \cite{GuidettiMathZ, Triebel}) the spaces $H_{\B}^{4\theta}(I)$  coincide with the complex interpolation spaces	
\begin{equation}\label{interpol}
H_\B^{4\theta}(I) =\big[ L_2(I), H_\B^{4}(I)\big]_\theta\ ,\quad \theta\in [0,1]\setminus\left\{\frac{1}{8}, \frac{3}{8}\right\}\ ,
\end{equation}
except for equivalent norms. To take into account the singular behavior of the right hand side of \eqref{u} as $u\to -1$, we further introduce, for $4\theta>2$ and $\kappa\in (0,1)$, the open subset
$$
S_{\theta}(\kappa):=\left\{v\in H_\B^{4\theta}(I)\,;\, \|v\|_{H_\B^{4\theta}(I)}< 1/\kappa \;\;\text{ and }\;\; -1+\kappa< v(x) \text{ for } x\in I \right\}
$$
of $H_\B^{4\theta}(I)$ with closure 
$$
\overline{S}_{\theta}(\kappa)=\left\{v\in H_\B^{4\theta}(I)\,;\, \|v\|_{H_\B^{4\theta}(I)}\le 1/\kappa \;\;\text{ and }\;\; -1+\kappa\le v(x) \text{ for } x\in I \right\}\ .
$$
The following proposition collects the properties of the solutions to the elliptic problem \eqref{psi}-\eqref{psibc} and is the main ingredient to investigate the parabolic problem for the deflection $u$.

\begin{prop}\label{L1}
Let $4\theta>2$ and $\kappa\in (0,1)$. For each $v\in\overline{S}_{\theta}(\kappa)$ there exists a unique solution \mbox{$\psi_v\in H^2(\Omega(v))$} to 
\begin{align*}
\ve^2 \,\partial_x^2\psi+\partial_z^2\psi &=0 \ ,& (x,z)\in \Omega(v)\ ,\\
\psi(x,z)&=\frac{1+z}{1+v(x)}\ ,& (x,z)\in\partial\Omega(v)\ ,
\end{align*}
with $\Omega(v) = \left\{(x,z)\,;\, -1<x<1\,,\, -1<z<v(x)\right\}$. Moreover, for $4\sigma\in [0,1/2)$, the mapping
$$
g: S_{\theta}(\kappa)\longrightarrow H_\B^{4\sigma}(I)\ ,\quad v\longmapsto \ve^2\vert\partial_{x}\psi_v(\cdot,v)\vert^2+\vert\partial_z\psi_v(\cdot,v)\vert^2
$$
is analytic, bounded, and uniformly Lipschitz continuous. If $v\in S_{\theta}(\kappa)$ is even, then so is $g(v)$.
\end{prop}

\begin{proof}
The proof is performed by transforming the problem to a fixed rectangle and using elliptic regularity theory.
A complete proof is contained in \cite[Proposition~5]{ELW1} with the additional use of the embedding $H_\B^{4\theta}(I)\hookrightarrow W_q^2(I)$ when chosing $q\in (2,\infty)$ such that $4\theta>5/2-1/q$.
\end{proof}

\medskip

To prove global existence of solutions for small voltage values $\lambda$ later on, it is useful to separate from the operator $\mathcal{K}$ in \eqref{AA} the 'cubic' third order term stemming from bending. Its regularity properties are stated in the next lemma:

\begin{lem}\label{hh}
Let $\kappa\in (0,1)$, $\delta\in (0,\kappa^{-1}]$, and $4\theta\in (3,4]$. Then the function
$$
h:S_\theta(\kappa)\rightarrow H^{4\theta-3}(I)\ , \quad v\longmapsto \frac{5}{2}\,\ve^2\,\beta\, \partial_{x}\left(\frac{\partial_xv(\partial_{x}^2v)^2}{(1+\ve^2(\partial_{x}v)^2)^{7/2}}\right)
$$
is analytic, and there is $c(\kappa)>0$ such that
$$
\| h(v_1)-h(v_2)\|_{H^{4\theta-3}(I)}\le c(\kappa)\, \delta^2\, \|v_1-v_2\|_{H^{4\theta}(I)}
$$
whenever $v_j\in S_\theta(\kappa)$ with $\|v_j\|_{H^{4\theta}(I)}\le \delta$ for $j=1,2$.
\end{lem}

\begin{proof}
This follows from the definition of $h$ and the fact that $H^{4\theta-2}(I)$ is an algebra with respect to pointwise multiplication when $4\theta>3$, for example, see \cite{AmannMultiplication}.
\end{proof}

Using the definition of the functions $g$ and $h$ from Proposition~\ref{L1} and Lemma~\ref{hh}, respectively, we are in a position to write \eqref{u}-\eqref{psibc} as a quasilinear Cauchy problem for the strip deflection~$u$ as follows:
\begin{equation}\label{CP}
\begin{split}
\partial_t u+ A(u)u
=-\lambda \, g(u)-h(u)\, \quad t>0\,,\qquad u(0)=u^0\, ,
\end{split}
\end{equation}
where, for a sufficiently smooth function $w$ on $I$, the linear operator $A(w)\in \mathcal{L}(H_D^4(I),L_2(I))$ is given by
\begin{equation}\label{AAa}
\begin{split}
A(w)v:=\beta\, & \partial_{x}^2\left(\frac{\partial_{x}^2v}{(1+\ve^2(\partial_{x}w)^2)^{5/2}}\right)
-\tau\, \partial_{x}\left(\frac{\partial_{x}v}{(1+\ve^2(\partial_{x}w)^2)^{1/2}}\right)\ ,\quad v\in H_D^4(I)\ .
\end{split}
\end{equation}
We next study the properties of $A(\cdot)$.  Given $\omega>0$ and $k\ge 1$, we let $\mathcal{H}(H_\B^4(I), L_2(I); k,\omega)$ denote the set of all $\mathcal{A}\in\mathcal{L}(H_\B^4(I), L_2(I))$ such that $\omega+\mathcal{A}$ is an isomorphism from $H_\B^4(I)$ onto $L_2(I)$ and satisfies the resolvent estimates
$$
\frac{1}{k}\,\le\,\frac{\|(\mu+\mathcal{A})v\|_{L_2(I)}}{\vert\mu\vert \,\| v\|_{L_2(I)}+\|v\|_{H_\B^4(I)}}\,\le \, k\ ,\qquad Re\, \mu\ge \omega\ ,\quad v\in H_\B^4(I)\setminus\{0\}\ .
$$
Then $\mathcal{A}\in\mathcal{H}(H_\B^4(I), L_2(I);k,\omega)$ implies $\mathcal{A}\in\mathcal{H}(H_\B^4(I), L_2(I))$, that is, $-\mathcal{A}$  generates an analytic semigroup on $L_2(I)$, see \cite[I.Theorem~1.2.2]{LQPP}.  The quasilinear operator $A$ from \eqref{AAa} enjoys the following properties.

\begin{lem}\label{SG} 
Given $4\theta>7/2$ and $\kappa\in (0,1)$, there are $k:=k(\kappa)\ge 1$ and $\omega:=\omega(\kappa)>0$ such that, for each $w\in \overline{S}_\theta(\kappa)$, the linear operator $A(w)$ defined in \eqref{AAa} is such that
$$
-2\omega+A(w)\in \mathcal{H}(H_\B^4(I), L_2(I);k,\omega)\ .
$$
Moreover, there is a constant $\ell(\kappa)>0$ such that
\begin{equation}\label{esti}
\| A(w_1)-A(w_2)\|_{\ml(H_\B^4(I), L_2(I))}\le \ell(\kappa) \, \|w_1-w_2\|_{H^{4\theta}(I)}\ ,\quad w_1, w_2\in \overline{S}_\theta(\kappa)\ .
\end{equation}
\end{lem}

\begin{proof}[{\bf Proof}]
Let $w\in \overline{S}_\theta(\kappa)$ be fixed and put $W:= 1+\ve^2(\partial_x w)^2$. Notice that $W\in C^2([-1,1])$ thanks to the continuous embedding of $H^{4\theta}(I)$ in $C^3([-1,1])$ and that $-A(w)$ reads
$$
-A(w) v = - \beta \partial_x^2\left( W^{-5/2} \partial_x^2 v \right) + \tau \partial_x\left( W^{-1/2} \partial_x v \right)\ , \qquad v\in H_D^4(I)\ .
$$
We then claim that $-A(w)$ is the generator of an analytic semigroup on $L_2(I)$. Indeed, consider its principal part $-\mathcal{A}_*$ defined by $
\mathcal{A}_* v:= \beta W^{-5/2} \partial_x^4 v$ and the boundary operator $\mathcal{B}_* v(\pm 1):= \big( v(\pm 1), -i\partial_x v(\pm 1)\big)$. Noticing that $W(\pm 1)=1$, it is readily seen that $(\mathcal{A}_*, \mathcal{B}_*)$ is normally elliptic in the sense of  \cite[Remark~4.2(b)]{AmannTeubner} and thus $-\mathcal{A}_* \vert_{H_\B^4(I)}$ is the generator of an analytic semigroup on $L_2(I)$. Now, since $-A(w)$ can be considered as a lower order perturbation of $-\mathcal{A}_* \vert_{H_\B^4(I)}$, the same is true for $-A(w)$, see \cite[I.Theorem~1.3.1]{LQPP} for instance.

Next, since $-A(w)$ has a compact resolvent, its spectrum consists of eigenvalues only. If $\mu\in \C$ is such an eigenvalue  with a corresponding eigenfunction $\varphi\in H_\B^4(I,\C)$, then testing the identity
$$
\mu \varphi+ A(w)\varphi=0
$$ 
with its complex conjugate $\overline{\varphi}$ shows that necessarily $\mu\in\R$ with 
$$
\mu\le-( \beta \mu_1^2 \|W\|_\infty^{-5/2}+\tau\mu_1\|W\|_\infty^{-1/2})=:-2\varsigma(w)\ ,
$$ 
where $\mu_1 =\pi^2/4$ is the principal eigenvalue of $-\partial_x^2$ with homogeneous Dirichlet boundary conditions in $I$. Thus, setting 
$$
\omega=\omega(\kappa):= \inf_{w\in \overline{S}_\theta(\kappa)} \varsigma(w)\, >0\ ,
$$ 
then, for each $w\in \overline{S}_\theta(\kappa)$, the half plane $\{ \mu\in\C\ ;\ \mathrm{Re}\, \mu>-2\omega\}$ is contained in the resolvent set of $-A(w)$, and, in particular, $-\omega+A(w)$ is an isomorphism from $H_\B^4(I)$ onto $L_2(I)$. Moreover, given $w\in \overline{S}_\theta(\kappa)$, $f\in L_2(I,\C)$, and $\mu\in\C$ with $\mathrm{Re}\,\mu >0$, the equation
$$
\mu v-2\omega v+ A(w)v=f
$$  
has a unique solution $v\in H_D^4(I)$. Testing this identity by the complex conjugate $\bar{v}$ of $v$, one  derives a resolvent estimate of the form
$$
\|(\mu-2\omega+A(w))^{-1}\|_{\ml(L_2(I))} \le \frac{c(\kappa)}{\vert\mu\vert} \ ,\quad w\in \overline{S}_\theta(\kappa)\ .
$$
Since obviously
$$
\|-2\omega+A(w)\|_{\ml(H_\B^4(I),L_2(I))}\le c_1(\kappa)\ ,\quad w\in \overline{S}_\theta(\kappa)\ ,
$$
it follows from \cite[I.Remark~1.2.1(a)]{LQPP} that 
$$
-2\omega+A(w)\in \mathcal{H}(H_\B^4(I),L_2(I);k,\omega)
$$
for some $k=k(\kappa)\ge 1$. Finally, the Lipschitz continuity \eqref{esti} follows from the definitions of $A(w)$ and $\overline{S}_\theta(\kappa)$ and the continuous embedding of $H^{4\theta}(I)$ in $C^3([-1,1])$ by straightforward computations.
\end{proof}

Clearly, the lower bound $w\ge -1+\kappa$ on $I$ for $w\in \overline{S}_\theta (\kappa)$ is not needed for Lemma~\ref{SG} to hold true, but is introduced for easier notation later on.

Now, if $w$ is a time-dependent function, then the solution $v$ to the linear Cauchy problem
$$
\dfrac{\rd}{\rd t} v + A(w(t))v = f(t)\ ,\quad t>s\ ,\qquad v(s)=v^0\ ,
$$
can no longer be expressed by a variation-of-constant formula involving semigroups. The representation formula in this case rather relies on the construction of a parabolic evolution operator $U_{A(w)}$ and reads
$$
v(t)= U_{A(w)}(t,s) v^0 + \int_s^t U_{A(w)}(t,r) f(r)\,\rd r\ ,\quad t>s\ .
$$
According to \cite{LQPP}, the construction of a (unique) parabolic evolution operator $U_{A(w)}$ is possible if $w$ is H\"older continuous in $t$. More precisely:

\begin{prop}\label{ES}
Let $\kappa\in (0,1)$ and $4\theta>7/2$. Let $\ell(\kappa)>0$ be as in Lemma~\ref{SG} and, given $\rho\in (0,1)$ and $N,T>0$, define
\begin{equation*}
\begin{split}
\mathcal{W}_T(\kappa):=\{w\in C([0,T],H_\B^{4\theta}(I))\,;\,\|&w(t)-w(s)\|_{H_\B^{4\theta}(I)}\le\frac{N}{\ell(\kappa)} \vert t-s\vert^\rho\\
&\text{and  } w(t)\in \overline{S}_\theta(\kappa)\ \text{for } 0\le t,s\le T\}\ ,
\end{split}
\end{equation*}
which is a complete metric space for the distance 
$$
\mathrm{dist}_{\mathcal{W}_T(\kappa)}(v,w) := \max_{t\in [0,T]}\left\{ \| v(t)-w(t)\|_{H_\B^{4\theta}(I)} \right\}\ , \qquad v,w \in \mathcal{W}_T(\kappa)\ ,
$$
induced by the uniform topology of $C([0,T],H_\B^{4\theta}(I))$. Then, there is a constant $c_0(\rho,\kappa)>0$ independent of $N$ and $T$ such that the following is true: for each  $w\in \mathcal{W}_T(\kappa)$, there exists a unique parabolic evolution operator $U_{A(w)}(t,s)$, $0\le s\le t\le T$, and
$$
\|U_{A(w)}(t,s)\|_{\mathcal{L}(H_\B^{4\alpha}(I),H_\B^{4\beta}(I))} \le c_*(\kappa)\, (t-s)^{\alpha-\beta}\, e^{-\vartheta (t-s)}\ ,\quad 0\le s < t\le T\ ,
$$
for $0\le \alpha\le\beta\le 1$ with $4\alpha, 4\beta\notin\{ 1/2, 3/2\}$ with a constant $c_*(\kappa)\ge 1$ depending on $N$, $\alpha$, and $\beta$ but independent of~$T$, and
\begin{equation}\label{vartheta}
-\vartheta:=c_0(\rho,\kappa) N^{1/\rho}-\omega(\kappa)\ ,
\end{equation} 
where $\omega(\kappa)>0$ stems from Lemma~\ref{SG}.
\end{prop}

\begin{proof}[{\bf Proof}]
Since
\begin{equation}\label{lqpp1}
A(w)\in C^\rho([0,T],\ml(H_\B^{4}(I),L_2(I)))\ ,\quad -2\omega(\kappa) + A(w)\subset \mathcal{H}(H_\B^{4}(I),L_2(I);k(\kappa),\omega(\kappa))
\end{equation}
with 
\begin{equation}\label{lqpp2}
\sup_{0\le s<t\le T}\frac{\| A(w(t))-A(w(s))\|_{\ml(H_\B^{4}(I),L_2(I))}}{\vert t-s\vert^\rho}\le N
\end{equation}
for each  $w\in \mathcal{W}_T(\kappa)$ by Lemma~\ref{SG}, the assertion follows from \cite[II.Theorem~5.1.1\& Lemma~5.1.3]{LQPP} and the interpolation result \eqref{interpol}.
\end{proof}

We are now in a position to prove the well-posedness of \eqref{CP}. {

\begin{prop}[{\bf Local Well-Posedness}]\label{Alocal}
 Given $4\xi\in (7/2,4]$, consider an initial condition $u^0\in  H_\B^{4\xi}(I)$ such that $u^0(x)>-1$ for $x\in I$. Then, for each voltage value $\lambda>0$, there is a unique  solution $(u,\psi)$ to \eqref{u}-\eqref{psibc} on the maximal interval of existence $[0,T_m)$ in the sense that
$$
u\in C^1\big((0,T_m),L_2(I)\big)\cap C\big((0,T_m), H_\B^4(I)\big) \cap C\big([0,T_m), H_\B^{4\xi}(I)\big)
$$
satisfies \eqref{u}-\eqref{ic} together with
$$
u(t,x)>-1\ ,\quad (t,x)\in [0,T_m)\times I\ , 
$$ 
and $\psi(t)\in H^2\big(\Omega(u(t))\big)$  solves \eqref{psi}-\eqref{psibc} in $\Omega(u(t))$ for each $t\in [0,T_m)$. In addition, if $\xi=1$, then $u\in C^1\big([0,T_m),L_2(I)\big)$.
\end{prop}

\begin{proof}
Let $4\xi\in (7/2,4]$ and  consider an initial condition $u^0\in  H_\B^{4\xi}(I)$ such that $u^0(x)>-1$ for $x\in I$. Choose $4\theta\in (7/2,4\xi)$ and set $\rho:=(\xi-\theta)/2>0$. Clearly, there is $\kappa\in (0,1/2)$ such that
\begin{equation}\label{44}
u^0\in \overline{S}_\theta(2\kappa) \cap \overline{S}_\xi(2\kappa)\ .
\end{equation}
Let $\omega(\kappa)>0$ and $c_0(\rho,\kappa)>0$ be as in Lemma~\ref{SG} and Proposition~\ref{ES}, respectively,
 and choose then $N>0$ such that
\begin{equation}
-\vartheta:=c_0(\rho, \kappa) N^{1/\rho}-\omega(\kappa)<0 \label{thetabis}
\end{equation}
in \eqref{vartheta}. Moreover, fix $4\sigma\in (0,1/2)$ and note that Proposition~\ref{ES} ensures that, for each $w\in\mathcal{W}_T(\kappa)$, the corresponding parabolic evolution operator $U_{A(w)}$ satisfies
\begin{equation}\label{UA}
\|U_{A(w)}(t,s)\|_{\mathcal{L}(H_\B^{4\theta}(I))} + (t-s)^{\theta-\sigma}\, \|U_{A(w)}(t,s)\|_{\mathcal{L}(H_\B^{4\sigma}(I),H_\B^{4\theta}(I))} \le c_*(\kappa)\, \, e^{-\vartheta (t-s)}\ ,
\end{equation}
for $0\le s \le t\le T$, where the constant $c_*(\kappa)\ge 1$ is independent of $w$ and $T>0$. 
With this notation, we are now in a position to set up the fixed point argument. Given $\delta\in (0,1/\kappa]$, we introduce a subset of $\mathcal{W}_T(\kappa)$ defined by 
$$
\mathcal{W}_T(\kappa,\delta):=\{w\in \mathcal{W}_T(\kappa)\,;\, \|w(t)\|_{H_\B^{4\theta}(I)}\le\delta\}\ ,
$$
and note that $\mathcal{W}_T(\kappa,\kappa^{-1})=\mathcal{W}_T(\kappa)$ (the role of $\delta$ will become clear later when addressing global existence issues in the proof of Corollary~\ref{Aglobalc}). Then Proposition~\ref{L1} and Lemma~\ref{hh} guarantee the existence of a constant $c_2(\kappa)>0$ independent of $T>0$ such that
\begin{equation}\label{10}
\|g(v (t))-g(w(t))\|_{H_\B^{4\sigma}(I)} + \frac{1}{\delta^2}\,\|h(v(t))-h(w(t))\|_{H_\B^{4\sigma}(I)}\le c_2(\kappa)\, \mathrm{dist}_{\mathcal{W}_T(\kappa)}(v,w)
\end{equation}
and
\begin{equation}\label{10a}
\| g(v(t))\|_{H_\B^{4\sigma}(I)}+ \frac{1}{\delta^3}\| h(v(t))\|_{H_\B^{4\sigma}(I)}\le c_2(\kappa)
\end{equation}
for $0\le t\le T$ whenever $v,w\in\mathcal{W}_T(\kappa,\delta)$ (recall that $4\sigma<1/2<4\theta - 3$ and $h(0)=0$).
Owing to the definition of parabolic evolution operators, the solution to \eqref{CP} is a fixed point of the map $\Lambda$ defined by
$$
\Lambda(v)(t):=U_{A(v)}(t,0)\,u^0-\int_0^t U_{A(v)}(t,s)\, \big(\lambda g\big(v(s)\big) +h(v(s)\big)\,\rd s\ ,\quad t\in [0,T]\ ,\quad v\in \mathcal{W}_T(\kappa)\ .
$$
We then claim that, for arbitrary $\lambda>0$, the map $\Lambda$ is a contraction from $\mathcal{W}_T(\kappa)$ into itself if $T=T(\kappa,\lambda)>0$ is sufficiently small as well as a contraction on $\mathcal{W}_T(\kappa,\delta)$ for any $T>0$ provided that both $\lambda$ and the initial condition $u^0$ are sufficiently small. To see this, consider $\delta\in (0,1/\kappa]$ and $v \in \mathcal{W}_T(\kappa,\delta)$. We infer from \eqref{lqpp1}, \eqref{lqpp2}, \eqref{10a}, the continuous embedding of $H_\B^{4\sigma}(I)$ in $L_2(I)$, and \cite[II.Theorem~5.3.1]{LQPP} that there is a constant  $m_*(\kappa)>0$ depending only on $\kappa$, $\xi$, and $\theta$ such that, for $0\le s\le t\le T$, 
\begin{equation}\label{122}
\begin{split}
\|\Lambda(& v)(t)-\Lambda(v)(s)\|_{H_\B^{4\theta}(I)}\\
&\le m_*(\kappa)\,   (t-s)^{2\rho}\, e^{-\vartheta t}\, \left(\|u^0\|_{H_\B^{4\theta+8\rho}(I)} + \|\lambda\,g(v)+h(v)\|_{L_\infty ((0,T),H_\B^{4\sigma}(I))}\right)\\
&\le m_*(\kappa)\, \left(\|u^0\|_{H_\B^{4\xi}(I)} +(\lambda+\delta^{3})\, c_2(\kappa)\right)\, (t-s)^{2\rho} e^{-\vartheta t} \\
&\le m_*(\kappa)\,  \left(\max_{0\le r\le T} r^{\rho}\, e^{-\vartheta r} \right)\, \left(\|u^0\|_{H_\B^{4\xi}(I)} +(\lambda+\delta^{3})\, c_2(\kappa)\right)\, (t-s)^{\rho} \ .
\end{split}
\end{equation}
We deduce in particular from \eqref{122} and the continuous embedding of $H^{4\xi}(I)$ in $H^{4\theta}(I)$ that, for $0\le t\le T$,
 \begin{equation}\label{122a}
\begin{split}
\|\Lambda(v)(t)\|_{H_\B^{4\theta}(I)}& \le\|\Lambda(v)(t)-\Lambda(v)(0)\|_{H_\B^{4\theta}(I)} +\|u^0\|_{H_\B^{4\theta}(I)}\\
&\le m_*(\kappa)\,   \left(\max_{0\le r\le T} r^{2\rho}\, e^{-\vartheta r}\right)\, \left(\|u^0\|_{H_\B^{4\xi}(I)} +(\lambda+\delta^{3})\, c_2(\kappa)\right)\, + c\|u^0\|_{H_\B^{4\xi}(I)}\ .
\end{split}
\end{equation}
Moreover, since $u^0\ge -1+2\kappa$ in $I$ by \eqref{44} and since $H_\B^{4\theta}(I)$ is continuously embedded in $L_\infty(I)$ with embedding constant, say, $c_3>0$, a further consequence of \eqref{122} is that, for $0\le t\le T$,
\begin{align}
\Lambda(v)(t)&=\  u^0-\big(\Lambda (v)(0)-\Lambda (v)(t)\big)\ge u^0- c_3\, \|\Lambda (v)(0)-\Lambda (v)(t)\|_{H_\B^{4\theta}(I)}\nonumber \\
&\ge  -1 + 2\kappa - c_3\,m_*(\kappa)\,   \left(\max_{0\le r\le T} r^{2\rho}\, e^{-\vartheta r}\right)\, \left(\|u^0\|_{H_\B^{4\xi}(I)} +(\lambda+\delta^{3})\, c_2(\kappa)\right) \label{zui}\ .
\end{align}
Next, if $v$ and $w$ are arbitrary elements in $\mathcal{W}_T(\kappa,\delta)$, it follows from \eqref{lqpp1}, \eqref{lqpp2}, the continuous embedding of $H_\B^{4\sigma}(I)$ in $L_2(I)$, and \cite[II.Theorem~5.2.1]{LQPP} that there is a constant $n_*(\kappa)>0$ depending only on $\kappa$, $\xi$, and $\theta$ such that, for $0\le t\le T$,
\begin{equation*}
\begin{split}
\|\Lambda(v)(t)&-\Lambda(w)(t)\|_{H_\B^{4\theta}(I)}\\
 &\le n_*(\kappa)\, e^{-\vartheta t}\,\Bigg\{ t^{1-\theta}\, \|\lambda\,(g(v)-g(w))+h(v)-h(w)\|_{L_\infty((0,t), H_\B^{4\sigma}(I))} \\
 & \qquad\qquad\qquad +t^{2\rho}\, \|A(v)-A(w)\|_{C([0,T],\ml(H_\B^{4}(I),L_2(I)))}\,\left[\|u^0\|_{H_\B^{4\xi}(I)} \right.\\
 &\qquad\qquad\qquad\qquad\qquad\qquad \left. \,+\, t^{1+\sigma-\theta-2\rho}\, \|\lambda\,g(v)+h(v)\|_{L_\infty((0,t), H_\B^{4\sigma}(I))}\right]     \Bigg\}\\
 &\le n_*(\kappa)\, t^{1-\theta}\, e^{-\vartheta t}\, c_2(\kappa)\,\big(\lambda+\delta^2\big)\,  \mathrm{dist}_{\mathcal{W}_T(\kappa)}(v,w)\\
 & \quad+n_*(\kappa)\,  t^{2\rho}\, e^{-\vartheta t} \, \ell(\kappa)\,  \mathrm{dist}_{\mathcal{W}_T(\kappa)}(v,w)\,\left[\|u^0\|_{H_\B^{4\xi}(I)}\,+\, t^{1+\sigma-\theta-2\rho}\, c_2(\kappa)\big(\lambda+\delta^3\big)\right]\ ,
\end{split}
\end{equation*}
where we have used \eqref{esti}, \eqref{10}, and \eqref{10a} for the second inequality. Thus there is $c_4(\kappa)>0$ such that 
\begin{equation}
\begin{split}\label{12}
\mathrm{dist}_{\mathcal{W}_T(\kappa)}(\Lambda(v),\Lambda(w))
 & \le c_4(\kappa)\, \left(\max_{0\le r\le T} r^{1-\theta}\, e^{-\vartheta r}\right)\, \big(\lambda+\delta^2\big)\,  \mathrm{dist}_{\mathcal{W}_T(\kappa)}(v,w) \\
 & \quad + \, c_4(\kappa)\,  \left(\max_{0\le r\le T} r^{2\rho}\, e^{-\vartheta r}\right) \, \|u^0\|_{H_\B^{4\xi}(I)}\, \mathrm{dist}_{\mathcal{W}_T(\kappa)}(v,w) \\
 & \quad +\, c_4(\kappa)\,\left(\max_{0\le r\le T} r^{1+\sigma-\theta}\, e^{-\vartheta r}\right) \, \big(\lambda+\delta^3\big)\, \mathrm{dist}_{\mathcal{W}_T(\kappa)}(v,w) \ .
\end{split}
\end{equation}
Now pick any $\lambda>0$ and take $\delta=\kappa^{-1}$. It follows from \eqref{122}-\eqref{12} and \eqref{44} that we may choose \mbox{$T=T(\kappa,\lambda)>0$} sufficiently small such that the mapping $\Lambda :\mathcal{W}_T(\kappa)\rightarrow \mathcal{W}_T(\kappa)$ defines a contraction for the distance $\mathrm{dist}_{\mathcal{W}_T(\kappa)}$ and thus has a unique fixed point $u$ in $\mathcal{W}_T(\kappa)$. Observe then that, owing to Proposition~\ref{L1}, Lemma~\ref{hh}, and Lemma~\ref{SG}, we have
$$
\big[ t \mapsto \big( A(u(t)), \lambda g(u(t)) + h(u(t)) \big) \big] \in C^\rho\big([0,T],\mathcal{H}(H_\B^{4}(I),L_2(I))\times H_\B^{4\sigma}(I)\big)
$$
with $\sigma>0$, and $u$ is a mild solution to \eqref{CP} on $[0,T]$ with $u^0\in H_\B^{4\xi}(I)$. Thus,
$$
u\in C^1\big((0,T],L_2(I)\big)\cap C\big((0,T], H_\B^4(I)\big) \cap C\big([0,T], H_\B^{4\xi}(I)\big)
$$
is a solution to \eqref{CP} by \cite[Theorem~10.1]{AmannTeubner}, which can be extended to some maximal interval $[0,T_m)$. Note that, when taking $\xi=1$, we obtain $u\in C\big([0,T_m), H_\B^{4}(I)\big)$ by a further use of \cite[Theorem~10.1]{AmannTeubner} and thus $u\in C^1\big([0,T_m),L_2(I)\big)$ by \eqref{CP}. This proves Proposition~\ref{Alocal} after setting $\psi(t,\cdot) := \psi_{u(t)}(\cdot) \in H^2(\Omega(u(t)))$ for $t\in [0,T_m)$, the latter being defined in Proposition~\ref{L1}. 
\end{proof}

We now supplement Proposition~\ref{Alocal} with a global existence result for small voltage values $\lambda$ and a criterion guaranteeing global existence.

\begin{cor}[{\bf Global Well-Posedness}]\label{Aglobalc}
Given $4\xi\in (7/2,4]$ and an initial condition $u^0\in  H_\B^{4\xi}(I)$ such that $u^0(x)>-1$ for $x\in I$, let $(u,\psi)$ be the solution to \eqref{u}-\eqref{psibc} on the maximal interval of existence $[0,T_m)$ provided by Proposition~\ref{Alocal}. Then, the following are true:

\begin{itemize}

\item[(i)] If, for each $T>0$, there is $\kappa(T)\in (0,1)$ such that $$\|u(t)\|_{H^{4\xi}(I)}\le \kappa(T)^{-1}\ ,\qquad u(t)\ge -1+\kappa(T)\ \text{ in } I$$ for $t\in [0,T_m)\cap [0,T]$, then the solution $(u,\psi)$ exists globally in time, that is, $T_m=\infty$.

\item[(ii)] Given $\kappa\in (0,1)$, there are $\lambda_*(\kappa)>0$ and $q(\kappa)>0$ such that the solution $(u,\psi)$ exists globally in time provided that $\lambda\in (0,\lambda_*(\kappa))$ and $\|u^0\|_{H^{4\xi}(I)}\le q(\kappa)$. Moreover, in this case there is $\rho>0$ such that  $u\in BUC^\rho([0,\infty),H_\B^{4\xi}(I))$ with $u(t)\ge -1+\kappa$ in $I$ for $t\ge 0$. 
\end{itemize}
\end{cor}

\begin{proof}
Since $T$ in the proof of Proposition~\ref{Alocal} depends only on $\kappa$ and $\lambda$, we readily obtain part~(i). It remains to prove part~(ii). To this end we note that, since the parameter $\vartheta$ defined in \eqref{thetabis} is positive, there are $\lambda_*(\kappa)>0$, $M(\kappa)>0$, and $\delta=\delta(\kappa)\le\kappa^{-1}$ sufficiently small such that \eqref{122}-\eqref{12} guarantee that the mapping $\Lambda :\mathcal{W}_T(\kappa,\delta)\rightarrow \mathcal{W}_T(\kappa,\delta)$ defines a contraction \emph{for each} $T>0$ provided that  $\lambda\in (0,\lambda_*(\kappa))$ and $u^0\in H_\B^{4\xi}(I)$ with $u^0\ge -1+2\kappa$ in $I$ and  $\|u^0\|_{H_\B^{4\xi}(I)}\le M(\kappa)$. Thus, in this case as well there is a unique fixed point $u$ of $\Lambda$ belonging to $\mathcal{W}_T(\kappa,\delta)$ for each $T>0$. By definition of $\mathcal{W}_T(\kappa,\delta)$, this implies (ii).
\end{proof}

Combining Proposition~\ref{Alocal} and Corollary~\ref{Aglobalc} (with $\xi=1$) we deduce Theorem~\ref{Aquasi}.  We end this section with an immediate consequence of Proposition~\ref{L1}, the uniqueness statement of Proposition~\ref{Alocal}, and the invariance of the equations with respect to the symmetry $(x,z)\to (-x,z)$.

\begin{cor}\label{Asim} 
If $u^0=u^0(x)$ in Theorem~\ref{Alocal}
is even with respect to $x\in I$, then, {for all $t\in [0,T_m)$}, $u=u(t,x)$ and \mbox{$\psi=\psi(t,x,z)$} are even with respect to $x\in I$ as well.
\end{cor}

\section{Steady-State Solutions: Proof of Theorem~\ref{TStable}}\label{s5}

We now give a more precise statement of the existence of steady-state solutions. The proof of the following proposition relies on the implicit function theorem and the principle of linearized stability as well as on the already established regularity of the nonlinearities. 

\begin{prop}[{\bf Steady-State Solutions}]\label{TStable}
Let $\kappa\in (0,1)$. There is $\delta=\delta(\kappa)\in (0,1/\ve]$ small enough such that:

\begin{itemize}

\item[(i)] There is an analytic function $[\lambda\mapsto U_\lambda]:[0,\delta(\kappa))\rightarrow H_\B^4(I)$ such that $(U_\lambda,\Psi_\lambda)$ is the unique steady state to
\eqref{u}-\eqref{psibc} satisfying 
$$
\|U_\lambda\|_{H_\B^4(I)}\le 1/\kappa \;\;\text{ with }\;\; -1+\kappa\le U_\lambda\le 0 \;\text{ in }\; I\ ,
$$ 
and $\Psi_\lambda\in H^2(\Omega(U_\lambda))$ when $\lambda\in (0,\delta)$. Moreover, $U_\lambda$ is even with $U_0=0$ and $\Psi_\lambda=\Psi_\lambda(x,z)$ is even with respect to $x\in I$.

\item[(ii)] Let $\lambda\in (0,\delta(\kappa))$. There are $\omega_0,m,R>0$ such that for each initial condition \mbox{$u^0\in H_\B^4(I)$} satisfying 
$$
\|u^0-U_\lambda\|_{H_\B^4(I)} <m\ ,
$$  
there is a unique global solution $(u,\psi)$ to \eqref{u}-\eqref{psibc} with
$$
u\in C^1\big([0,\infty),L_2(I)\big)\cap C\big([0,\infty), H_\B^4(I)\big)\ ,\qquad \psi(t)\in H^2\big(\Omega(u(t))\big)\ ,\quad t\ge 0\ ,
$$
and
$u(t)>-1$ in $I$ for each $t\ge 0$. Moreover,
\begin{equation}\label{est}
\|u(t)-U_\lambda\|_{H_\B^4(I)}+\|\partial_t u(t)\|_{L_{2}(I)} \le R e^{-\omega_0 t} \|u^0-U_\lambda\|_{H_\B^4(I)}\ ,\quad t\ge 0\ .
\end{equation}
\end{itemize}
\end{prop}

\begin{proof}
Since the operator $A(v)\in\mathcal{L}(H_\B^4(I),L_2(I))$ defined in \eqref{AAa} is invertible for $v\in S_1(\kappa)$ according to Lemma~\ref{SG}, we may define the mapping
$$
F:\R\times S_1(\kappa)\rightarrow H_\B^4(I)\ ,\quad (\lambda,v)\mapsto v + A(v)^{-1}(\lambda g(v)+h(v))\ .
$$
Then $F$ is analytic by Proposition~\ref{L1} and Lemma~\ref{hh} with $F(0,0)=0$. Moreover, since the Fr\'echet derivative of $h$ vanishes at zero, we have
 $D_vF(0,0) =\mathrm{id}_{H_\B^4(I)}$. Now, the implicit function theorem ensures the existence of $\delta_1 = \delta_1(\kappa) \in (0,1/\ve]$ and an analytic function 
$$
[\lambda\mapsto U_\lambda]: (-\delta_1,\delta_1)\rightarrow H_\B^4(I)
$$ 
such that $F(\lambda,U_\lambda)=0$ for $\lambda\in [0,\delta_1)$. Introducing $\Psi_\lambda := \psi_{U_\lambda}\in H^2(\Omega(U_\lambda))$ according to Proposition~\ref{L1}, $(U_\lambda,\Psi_\lambda)$ is the unique steady state to \eqref{u}-\eqref{psibc} satisfying $U_\lambda\in S_1(\kappa)$. The non-positivity of $U_\lambda$ readily follows from \cite[Section~4.2]{LaurencotWalker_JAM} and the fact that $g(U_\lambda)\ge 0$. That $U_\lambda=U_\lambda(x)$ and $\Psi_\lambda=\Psi_\lambda(x,z)$ is even in $x\in I$ are immediate consequences of uniqueness and the invariance of the equations with respect to the symmetry $(x,z)\to (-x,z)$. This proves Proposition~\ref{TStable}(i).  

\medskip

It remains to prove the asymptotic stability of the steady states $(U_\lambda,\Psi_\lambda)$ as claimed in Proposition~\ref{TStable}(ii). We proceed similarly as in \cite[Theorem~3(ii)]{ELW1} by using the principle of linearized stability.  Let $\lambda\in (0,\delta_1(\kappa))$ and define $Q\in C^\infty(S_1(\kappa),L_2(I))$ by
$$
Q(u):=-A(u)u-\lambda g(u)-h(u)\ ,
$$
which satisfies in particular $Q(U_\lambda)=0$. The linearization of \eqref{CP} around $U_\lambda$ then reads
$$
\frac{\rd}{\rd t}v-D_uQ(U_\lambda)[v]=G_\lambda(v) := Q(v+U_\lambda)-D_uQ(U_\lambda)[v]\ ,
$$
so that $u=v+U_\lambda$ solves \eqref{CP}. According to Proposition~\ref{L1}, the map $G_\lambda\in C^\infty(O_\lambda,L_2(I))$ is defined on some open neighborhood $O_\lambda$ of zero in $H_\B^4(I)$ such that $U_\lambda+O_\lambda\subset S_1(\kappa)$. In view of \eqref{CP}  we obtain
\begin{equation}\label{lin}
\frac{\rd}{\rd t} v + \big( A(U_\lambda) + B_\lambda \big) v = G_\lambda(v)\ ,
\end{equation}
where
\begin{equation*}
\begin{split}
B_\lambda v:&=\ \lambda\, D_ug(U_\lambda) [v] + D_u A(U_\lambda) [v] U_\lambda + D_u h(U_\lambda) [v]\,,\quad v\in H_\B^4(I)\,,
\end{split}
\end{equation*}
satisfies
$$
\|B_\lambda\|_{\mathcal{L}(H_\B^4(I),L_2(I))}\le c(\kappa)\,\big(\lambda+\|U_\lambda\|_{H_\B^4(I)}^2\big)
$$
for some $c(\kappa)>0$ independent of $\lambda$ small, say, $\lambda\in [0,\delta_1/2]$. The continuity of $\lambda\mapsto U_\lambda$ then implies 
$$
\|B_\lambda\|_{\mathcal{L}(H_\B^4(I),L_2(I))}
\rightarrow 0\quad \text{as}\quad \lambda\rightarrow 0\ .
$$
Since $U_\lambda\in S_1(\kappa)$, we have $-2\omega(\kappa) + A(U_\lambda)\in \mathcal{H}(H_\B^4(I),L_2(I); k(\kappa),\omega(\kappa))$ by Lemma~\ref{SG}. In particular, $-A(U_\lambda)$ is the generator of an analytic semigroup on $L_2(I)$ with spectral bound not exceeding \mbox{$-\omega(\kappa)<0$}.
 Consequently, it follows from \cite[I.Proposition~1.4.2]{LQPP} that $-(A(U_\lambda)+B_\lambda)$ is the generator of an analytic semigroup on $L_2(I)$ and there is $\omega_1=\omega_1(\kappa)>0$ such that the complex half plane  $\{ z\in\C\ ;\ \mathrm{Re}\, z\ge -\omega_1\}$ is included in the resolvent set of $-(A(U_\lambda)+B_\lambda)$ provided that $\lambda>0$ is sufficiently small. 
Now we may apply \cite[Theorem~9.1.2]{Lunardi} and conclude that the statement~(ii) of Proposition~\ref{TStable} holds true for $\lambda\in (0,\delta(\kappa))$ for some $\delta(\kappa)>0$ possibly smaller than $\delta_1$.
\end{proof}

\section*{Acknowledgments}

The work of Ph.L. was partially supported by the CIMI (Centre International de Ma\-th\'e\-ma\-ti\-ques et d'Informatique) Excellence program and by the Deutscher Akademischer Austausch Dienst (DAAD) while enjoying the hospitality of the Institut f\"ur Angewandte Mathematik, Leibniz Universit\"at Hannover.




\begin{thebibliography}{1}

\bibitem{AmannMultiplication}
H.~Amann. \textit{Multiplication in Sobolev and Besov spaces.} In {\em Nonlinear analysis}, Scuola Norm. Sup. di Pisa Quaderni, 27--50, Scuola Norm. Sup., Pisa, 1991.

\bibitem{AmannTeubner}
H.~Amann. \textit{Nonhomogeneous Linear and Quasilinear Elliptic and Parabolic Boundary Value Problems.} In {\em Function Spaces, Differential Operators and Nonlinear Analysis}, H.~Schmeisser, H.~Triebel (eds.), Teubner-Texte zur Math. {\bf 133}, 9--126, Teubner, Stuttgart, Leipzig, 1993.

\bibitem{LQPP}
H.~Amann. \textit{Linear and Quasilinear Parabolic Problems,
Volume~I: Abstract Linear Theory.} Birkh\"auser, Basel, Boston, Berlin, 1995.

\bibitem{BrubakerPelesko_EJAM}
N.~Brubaker and J.~Pelesko. \textit{Non-linear effects on canonical MEMS models.} European J. Appl. Math {\bf 22} (2011), 455--470.

\bibitem{CFT}
D. Cassani, L. Fattorusso, and A. Tarsia. \textit{Nonlocal dynamic problems with singular nonlinearities and applications to MEMS.} Preprint (2013). 

\bibitem{Ci07}
G.~Cimatti. \textit{A free boundary problem in the theory of electrically actuated microdevices.} Appl. Math. Lett. \textbf{20} (2007), 1232--1236. 

\bibitem{ELW1}
J.~Escher, Ph.~Lauren\c{c}ot, and Ch.~Walker. \textit{A parabolic free boundary problem modeling electrostatic MEMS.} Arch. Ration. Mech. Anal., to appear.

\bibitem{ELW2}
J.~Escher, Ph.~Lauren\c{c}ot, and Ch.~Walker. \textit{Dynamics of a free boundary problem with curvature modeling electrostatic MEMS.} Preprint (2013).

\bibitem{ELW3}
J.~Escher, Ph.~Lauren\c{c}ot, and Ch.~Walker. \textit{Finite time singularity in a free boundary problem modeling MEMS.} Preprint  (2013).

\bibitem{EspositoGhoussoubGuo}
P.~Esposito, N.~Ghoussoub, and Y.~Guo. \textit{Mathematical Analysis of Partial Differential Equations Modeling Electrostatic MEMS.} Courant Lecture Notes in Mathematics \textbf{20}, Courant Institute of Mathematical Sciences, New York, 2010.

\bibitem{Grisvard}
P.~Grisvard. \textit{Elliptic Problems in Nonsmooth Domains.} Monographs and Studies in Mathematics {\bf 24}, Pitman (Advanced Publishing Program), Boston, MA, 1985.

\bibitem{GuidettiMathZ}
D.~Guidetti. \textit{On interpolation with boundary conditions.}
Math. Z. \textbf{207} (1991), 439--460. 

\bibitem{Gu10}
Y.~Guo. \textit{Dynamical solutions of singular wave equations modeling electrostatic MEMS.} SIAM J. Appl. Dyn. Syst. \textbf{9} (2010), 1135--1163.

\bibitem{GW09}
Z.~Guo and J.~Wei. \textit{On a fourth-order nonlinear elliptic equation with negative exponent.} SIAM J. Math. Anal. \textbf{40} (2009), 2034--2054.

\bibitem{HP05}
A.~Henrot and M.~Pierre. \textit{Variation et Optimisation de Formes.} Math\'ematiques $\&$ Applications (Berlin) {\bf 48} Springer, Berlin, 2005.

\bibitem{KLNT11}
N.I.~Kavallaris, A.A.~Lacey, C.V.~Nikolopoulos, and D.E.~Tzanetis. \textit{A hyperbolic non-local problem modelling MEMS technology.} Rocky Mountain J. Math. \textbf{41} (2011), 505--534.

\bibitem{LaurencotWalker_ARMA}
Ph.~Lauren\c{c}ot and Ch.~Walker. \textit{A stationary free boundary problem modeling electrostatic MEMS.} Arch. Ration. Mech. Anal.  {\bf 207} (2013), 139--158. 

\bibitem{LW_AnalPDE}
Ph.~Lauren\c{c}ot and Ch.~Walker. \textit{A free boundary problem modeling electrostatic MEMS:\ I. Linear bending effects.}
Preprint (2013).

\bibitem{LW_AHP}
Ph.~Lauren\c{c}ot and Ch.~Walker. \textit{A fourth-order model for MEMS with clamped boundary conditions.} Preprint (2013).

\bibitem{LaurencotWalker_JAM}
Ph.~Lauren\c{c}ot and Ch.~Walker. \textit{Sign-preserving property for some fourth-order elliptic operators in one dimension or in radial symmetry.} J. Anal. Math., to appear.

\bibitem{LinYang}
F.~Lin and Y.~Yang. \textit{Nonlinear non-local elliptic equation modelling electrostatic actuation.} Proc. R. Soc. Lond. Ser. A Math. Phys. Eng. Sci. \textbf{463} (2007), 1323--1337.
 
\bibitem{LL12}
A.E.~Lindsay and J.~Lega. \textit{Multiple quenching solutions of a fourth order parabolic PDE with a singular nonlinearity modeling a MEMS capacitor.} SIAM J. Appl. Math. \textbf{72} (2012), 935--958.

\bibitem{Lunardi}
A.~Lunardi. \textit{Analytic Semigroups and Optimal Regularity in Parabolic Problems.} Progress in Nonlinear Differential Equations and their Applications {\bf 16}, Birkh\"auser Verlag, Basel, 1995.

\bibitem{Necas67}
J.~Ne{\v{c}}as. \textit{Les M\'ethodes Directes en Th\'eorie des Equations Elliptiques.} Masson et Cie, Paris, 1967.

\bibitem{P02}
J.A.~Pelesko. \textit{Mathematical modeling of electrostatic MEMS with tailored dielectric properties.} SIAM J. Appl. Math. \textbf{62} (2001), 888--908. 

\bibitem{BP03}
J.A.~Pelesko and D.H.~Bernstein. \textit{Modeling MEMS and NEMS.} Chapman \& Hall/CRC, Boca Raton, FL, 2003.

\bibitem{PT01}
J.A.~Pelesko and A.A.~Triolo. \textit{Nonlocal problems in MEMS device control.} J. Engrg. Math. \textbf{41} (2001), 345--366.

\bibitem {Triebel}
H.~Triebel. \textit{Interpolation Theory, Function Spaces,
Differential Operators.} Second edition, Johann Ambrosius Barth, Heidelberg, Leipzig, 1995.

\end{thebibliography}
\end{document}